\documentclass[reqno,12pt]{amsart}
\usepackage{amsmath,amsthm,amssymb}
\usepackage{fullpage}
\usepackage{xcolor}
\usepackage{stmaryrd}
\begingroup
\catcode`[=\active \catcode`]=\active
\gdef[{\llbracket} \gdef]{\rrbracket}
\def\getdelim#1#2#3#4\relax{"#4}
\global\delcode`[=\expandafter\getdelim\llbracket\relax
\global\delcode`]=\expandafter\getdelim\rrbracket\relax
\endgroup
\newtheorem{thm}{Theorem}[section]
\newtheorem{lem}[thm]{Lemma}

\theoremstyle{definition}
\newtheorem{defn}[thm]{Definition}

\newtheorem{ex}[thm]{Example}

\theoremstyle{remark} \numberwithin{equation}{section}

\begin{document}
\title[]{Sufficient Conditions for Existence of Positive Solutions for a Caputo Fractional Singular Boundary Value Problem}
\date{}
\author{Naseer Ahmad Asif}
\address{Department of Mathematics, School of Science, University of Management and Technology, C-II Johar Town, 54770 Lahore, Pakistan}%
\email{naseerasif@yahoo.com}%

\begin{abstract}
{We present sufficient conditions for the existence of positive solutions for a class of fractional singular boundary value problems in presence of Caputo fractional derivative. Further, the nonlinearity involved has singularity with respect to independent variable as well as with respect to dependent variable.}
\end{abstract}

\maketitle

\section{introduction}

Mathematical models involving fractional order derivatives offer better description of physical phenomena such as in mechanics, in control systems, fluid flow in porous media, signal and image processing, aerodynamics, electromagnetics, viscoelasticity \cite{bt,bbri,kt,luch}. Recently, the area of fractional order boundary value problems (BVPs) has achieved a great progress in respect of both theoretically and physical applications \cite{borai,tsc,dhelm,kilbas,lak4,miller,podlubny,maae}. Since most of the nonlinear fractional differential equations do not have exact analytic solution, therefore, results to establish existence of solutions have attracted attention of many researchers \cite{sz,yjx,zshl,zshl1,zshl2,zshz}.
However, few articles in literature have studied the existence of solution for singular BVPs of fractional order, see \cite{ars,mmz,ss,szhx,xjy,zmw}.


In this article, we establish criteria for positive existence of the following fractional order BVP
\begin{equation}\label{mp}\begin{split}
{}^{C}D_{0^{+}}^{\,\mu}x(t)+f(t,x(t))&=\omega\,x(t),\hspace{0.4cm}t\in(0,1),\hspace{0.4cm}1<\mu\leq2,\hspace{0.4cm}\omega>0,\\
x'(0)=0,\,\,x(1)&=0,
\end{split}\end{equation}
where ${}^{C}D_{0^{+}}^{\,\mu}$ Caputo fractional left derivative of order $\mu$,  $f:(0,1)\times(0,\infty)\rightarrow\mathbb{R}$ is continuous and singular at $t=0$, $t=1$ and $x=0$. We prove positive existence for BVP \eqref{mp} in the space $X:=\{x:x\in C[0,1],\,{}^{C}D_{0^{+}}^{\,\mu}\,x\in C(0,1)\}$. By positive solution $x$ of BVP \eqref{mp} we mean $x\in X$ satisfies BVP \eqref{mp} and $x(t)>0$ for $t\in[0,1)$.

The rest of the paper is organized as follows. In Section \ref{pre}, the definition of fractional derivative and some preliminaries lemmas are presented. In Section \ref{mr}, by the use of fixed-point theorem and results of functional analysis, the existence of positive solution has established. An example is presented to illustrate the main theorem.

\section{preliminaries}\label{pre}

\begin{defn} \cite{suzhang,zhang}
The Caputo fractional left derivative of a function $x\in AC^{n}[0,\infty)$ of order $\mu\in(n-1,n]$, $n\in\mathbb{N}$,  is
\begin{align*}{}^{C}D_{0^{+}}^{\mu}x(t)=\frac{1}{\Gamma(n-\mu)}\int_{0}^{t}\frac{x^{(n)}(\tau)}{(t-\tau)^{\mu-n+1}}d\tau.\end{align*}
\end{defn}

Further, the following Laplace transforms are essential for our work

\begin{equation}\label{lt}\begin{split}
\mathcal{L}\{({}^{C}D_{0^{+}}^{\,\mu}x)(t)\}(s)&=s^{\,\mu}\mathcal{L}\{x(t)\}(s)-\sum_{k=0}^{n-1}s^{\mu-k-1}x^{(k)}(0),\hspace{0.4cm}n-1<\mu\leq n,\\
\mathcal{L}\{t^{\nu-1}E_{\mu,\nu}(\omega t^{\mu})\}(s)&=\frac{s^{\mu-\nu}}{s^{\mu}-\omega},\hspace{0.4cm}\mu>0,\,\,\nu>0,\,\,\left|\frac{\omega}{s^{\mu}}\right|<1,
\end{split}\end{equation}

where $E_{\mu,\nu}(t):=\sum_{k=0}^{\infty}\frac{t^{k}}{\Gamma(\mu k+\nu)}$ is the modified Mittag-Leffler function.

\begin{lem}\label{lemir}
Let $y\in L[0,1]$, then the BVP
\begin{equation}\label{211}\begin{split}
{}^{C}D_{0^{+}}^{\,\mu}x(t)+y(t)&=\omega\,x(t),\hspace{0.4cm}t\in(0,1),\hspace{0.4cm}1<\mu\leq2,\hspace{0.4cm}\omega>0,\\
x'(0)=0,\,\,x(1)&=0,
\end{split}\end{equation}
has integral representation
\begin{equation}\label{xz}
x(t)=\int_{0}^{1}G(t,\tau)\,y(\tau)d\tau,\hspace{0.4cm}t\in[0,1],
\end{equation}
where
\begin{equation}\label{gf}
G(t,\tau)=\begin{cases}
\frac{E_{\mu,1}(\omega t^{\mu})}{E_{\mu,1}(\omega)}(1-\tau)^{\mu-1}E_{\mu,\mu}(\omega(1-\tau)^{\mu})-(t-\tau)^{\mu-1}E_{\mu,\mu}(\omega(t-\tau)^{\mu}),\,&0\leq\tau\leq t\leq1,\\
\frac{E_{\mu,1}(\omega t^{\mu})}{E_{\mu,1}(\omega)}(1-\tau)^{\mu-1}E_{\mu,\mu}(\omega(1-\tau)^{\mu}),\,&0\leq t\leq\tau\leq1.
\end{cases}\end{equation}
\end{lem}

\begin{proof}
Consider the extended differential equation
\begin{equation}\label{exd}
{}^{C}D_{0^{+}}^{\,\mu}x(t)+y_{*}(t)=\omega\,x(t),\hspace{0.4cm}t>0,\hspace{0.4cm}1<\mu\leq2,\hspace{0.4cm}\omega>0,
\end{equation}
where $y_{*}:(0,\infty)\rightarrow\mathbb{R}$ is defined as
\begin{align*}
y_{*}(t)=\begin{cases}
y(t),&0<t<1,\\
0,&t\geq 1.
\end{cases}
\end{align*}
Taking Laplace transform of \eqref{exd}, we have
\begin{align*}
\mathcal{L}\{{}^{C}D_{0^{+}}^{\,\mu}x(t)\}(s)+\mathcal{L}\{y_{*}(t)\}(s)=\omega \mathcal{L}\{x(t)\}(s)
\end{align*}
which in view of \eqref{lt}, leads to
\begin{align*}
\mathcal{L}\{x(t)\}(s)=\frac{s^{\mu-1}}{s^{\mu}-\omega}\,x(0)+\frac{s^{\mu-2}}{s^{\mu}-\omega}\,x'(0)-\frac{1}{s^{\mu}-\omega}\mathcal{L}\{y_{*}(t)\}(s)
\end{align*}
Taking inverse Laplace transform we have
\begin{align*}
x(t)=E_{\mu,1}(\omega t^{\mu})\,x(0)+tE_{\mu,2}(\omega t^{\mu})\,x'(0)-\int_{0}^{t}(t-\tau)^{\mu-1}E_{\mu,\mu}(\omega(t-\tau)^{\mu})y_{*}(\tau)d\tau,\hspace{0.4cm}t\geq0
\end{align*}
Now employing BCs \eqref{211}, we have
\begin{align*}
x(t)=\frac{E_{\mu,1}(\omega t^{\mu})}{E_{\mu,1}(\omega)}\int_{0}^{1}(1-\tau)^{\mu-1}E_{\mu,\mu}(\omega(1-\tau)^{\mu})y(\tau)d\tau-\int_{0}^{t}(t-\tau)^{\mu-1}E_{\mu,\mu}(\omega(t-\tau)^{\mu})y(\tau)d\tau,\hspace{0.4cm}t\in[0,1],
\end{align*}
which is equivalent to \eqref{xz}.
\end{proof}

\begin{lem}\label{gbound}
The Green's function \eqref{gf} satisfies
\begin{itemize}
\item[(1).] $G:[0,1]\times[0,1]\rightarrow[0,\infty)$ is continuous and is positive on $[0,1)\times[0,1)$;
\item[(2).] $G(t,\tau)\leq E_{\mu,\mu}(\omega)$, for $(t,\tau)\in[0,1]\times[0,1]$; and
\item[(3).] $\int_{0}^{1}G(t,\tau)d\tau=\frac{\sigma_{\mu,\omega}(t)}{\omega E_{\mu,1}(\omega)}$, for $t\in[0,1]$, $\sigma_{\mu,\omega}(t):=E_{\mu,1}(\omega t^{\mu})E_{\mu,\mu+1}(\omega)-t^{\mu}E_{\mu,1}(\omega)E_{\mu,\mu+1}(\omega t^{\mu})$.
\end{itemize}
\end{lem}

\begin{proof}
\begin{itemize}
\item[(1).] Clearly, Green's function $G(t,\tau)$ is continuous for $(t,\tau)\in[0,1]\times[0,1]$. Moreover, $G(t,\tau)>0$ for $(t,\tau)\in[0,1)\times[0,1)$. 
\item[(2).] For $\tau\in[0,1]$ we have $(1-\tau)^{\mu-1}\leq1$. Consequently, from \eqref{gf}, we have
\begin{align*}
G(t,\tau)\leq E_{\mu,\mu}(\omega),\text{ for }(t,\tau)\in[0,1]\times[0,1].
\end{align*}
\item[(3).] Integrating \eqref{gf} with respect to $\tau$ from $0$ to $1$, we have
\begin{align*}
\int_{0}^{1}G(t,\tau)d\tau&=\frac{E_{\mu,1}(\omega t^{\mu})}{E_{\mu,1}(\omega)}\int_{0}^{1}(1-\tau)^{\mu-1}E_{\mu,\mu}(\omega(1-\tau)^{\mu})d\tau-\int_{0}^{t}(t-\tau)^{\mu-1}E_{\mu,\mu}(\omega(t-\tau)^{\mu})d\tau\\
&=\frac{E_{\mu,1}(\omega t^{\mu})}{E_{\mu,1}(\omega)}\int_{0}^{1}\tau^{\mu-1}E_{\mu,\mu}(\omega\tau^{\mu})d\tau-\int_{0}^{t}\tau^{\mu-1}E_{\mu,\mu}(\omega\tau^{\mu})d\tau\\
&=\frac{E_{\mu,1}(\omega t^{\mu})E_{\mu,\mu+1}(\omega)-t^{\mu}E_{\mu,1}(\omega)E_{\mu,\mu+1}(\omega t^{\mu})}{\omega E_{\mu,1}(\omega)}=\frac{\sigma_{\mu,\omega}(t)}{\omega E_{\mu,1}(\omega)}
\end{align*}
\end{itemize}
\end{proof}

\section{main result}\label{mr}

Assume that
\begin{itemize}
\item[(A1).] There exist $q\in C(0,1)$, $u\in C(0,\infty)$ decreasing, and $v\in C[0,\infty)$ increasing such that
\begin{align*}|f(t,x)|\leq q(t)(u(x)+v(x)),\hspace{0.4cm}t\in(0,1),\hspace{0.4cm}x\in(0,\infty),\end{align*}
\begin{align*}\int_{0}^{1}q(t)dt<\infty,\text{ and }\int_{0}^{1}q(t)u\left(c\,\sigma_{\mu,\omega}(t)\right)dt<\infty\,\text{ for }c>0.
\end{align*}
\item[(A2).] There exist a constant $R>\frac{\gamma_{_{R}}\,E_{\mu,\mu+1}(\omega)}{\omega E_{\mu,1}(\omega)}$ such that, for $t\in(0,1)$ and $x\in(0,R]$, $f(t,x)\geq\gamma_{_{R}}$, where the parameter $\gamma_{r}$ is positive and decreasing for $r>0$. Moreover,
\begin{align*}\frac{R}{E_{\mu,\mu}(\omega)\,\chi_{_{R}}\left(1+\frac{q(R)}{p(R)}\right)}>1\end{align*}
where
\begin{align*}\chi_{r}=\int_{0}^{1}q(t)u\left(\frac{\gamma_{_{r}}\,\sigma_{\mu,\omega}(t)}{\omega E_{\mu,1}(\omega)}\right)dt.\end{align*}
\end{itemize}
In view of $(A2)$, choose $\varepsilon\in(0,R-\frac{\gamma_{_{R}}\,E_{\mu,\mu+1}(\omega)}{\omega E_{\mu,1}(\omega)}]$ such that
\begin{equation}\label{eps}
\frac{R-\varepsilon}{E_{\mu,\mu}(\omega)\,\chi_{_{R+\varepsilon}}\left(1+\frac{q(R+\varepsilon)}{p(R+\varepsilon)}\right)}\geq1
\end{equation}
For $m\in\mathbb{N}$ with $\frac{1}{m}<\varepsilon$, consider the modified BVP
\begin{equation}\label{sp}\begin{split}
{}^{C}D_{0^{+}}^{\,\mu}x(t)+f\left(t,\min\{\max\{x(t)+\frac{1}{m},\frac{1}{m}\},R\}\right)&=\omega x(t),\hspace{0.1cm}t\in(0,1),\hspace{0.1cm}1<\mu\leq2,\hspace{0.1cm}\omega>0,\\
x'(0)=0,\hspace{0.3cm}x(1)&=0,
\end{split}\end{equation}
which in view of Lemma \ref{lemir}, has integral representation
\begin{equation*}x(t)=\int_{0}^{1}G(t,\tau)f\left(\tau,\min\{\max\{x(\tau)+\frac{1}{m},\frac{1}{m}\},R\}\right)d\tau,\hspace{0.4cm}t\in[0,1].\end{equation*}
Define $T_{m}:C[0,1]\rightarrow C[0,1]$ by
\begin{equation}\label{mapt}T_{m}x(t)=\int_{0}^{1}G(t,\tau)f\left(\tau,\min\{\max\{x(\tau)+\frac{1}{m},\frac{1}{m}\},R\}\right)d\tau,\hspace{0.4cm}t\in[0,1].\end{equation}
Clearly, fixed points of $T_{m}$ are solutions of BVP \eqref{sp}.
\begin{thm}\label{mainth}
Assume that $(A1)$ and $(A2)$ hold. Then the BVP \eqref{mp} has a positive solution.
\end{thm}

\begin{proof}
In view of $(A1)$ and Schauder's fixed point theorem the map $T_{m}$ defined by \eqref{mapt} has a fixed point $x_{m}\in C[0,1]$. Thus
\begin{equation}\label{fp}
x_{m}(t)=\int_{0}^{1}G(t,\tau)f\left(\tau,\min\{\max\{x(\tau)+\frac{1}{m},\frac{1}{m}\},R\}\right)d\tau,\hspace{0.4cm}t\in[0,1]
\end{equation}
which in view of $(A2)$ and Lemma \ref{gbound}, leads to
\begin{equation}\label{lb}
x_{m}(t)\geq\int_{0}^{1}G(t,\tau)\gamma_{_{R}}d\tau\geq\frac{\gamma_{_{R+\varepsilon}}\,\sigma_{\mu,\omega}(t)}{\omega E_{\mu,1}(\omega)}
\end{equation}
Also \eqref{fp} in view of Lemma \ref{gbound}, $(A1)$, \eqref{lb} and \eqref{eps}, leads to
\begin{equation}\label{ub}\begin{split}
x_{m}(t)&\leq E_{\mu,\mu}(\omega)\int_{0}^{1}q(\tau)u(\min\{\max\{x_{m}(\tau)+\frac{1}{m},\frac{1}{m}\},R\})\left(1+\frac{v(\min\{\max\{x_{m}(\tau)+\frac{1}{m},\frac{1}{m}\},R\})}{u(\min\{\max\{x_{m}(\tau)+\frac{1}{m},\frac{1}{m}\},R\})}\right)d\tau\\
&\leq E_{\mu,\mu}(\omega)\int_{0}^{1}q(\tau)u\left(\frac{\gamma_{_{R+\varepsilon}}\,\sigma_{\mu,\omega}(\tau)}{\omega E_{\mu,1}(\omega)}\right)\left(1+\frac{v(R+\varepsilon)}{u(R+\varepsilon)}\right)d\tau\\
&=E_{\mu,\mu}(\omega)\,\chi_{_{R+\varepsilon}}\left(1+\frac{v(R+\varepsilon)}{u(R+\varepsilon)}\right)\leq R-\varepsilon
\end{split}\end{equation}
Consequently, from \eqref{lb} and \eqref{ub}, solution $x_{m}$ of BVP \eqref{sp} satisfies
\begin{equation}\label{xnt}x_{m}(t)=\int_{0}^{1}G(t,\tau)f\left(\tau,x_{m}(\tau)+\frac{1}{m}\right)d\tau,\hspace{0.4cm}t\in[0,1]\end{equation}
and
\begin{align*}
\frac{\gamma_{_{R+\varepsilon}}\,\sigma_{\mu,\omega}(t)}{\omega E_{\mu,1}(\omega)}\leq x_{m}(t)<R,\hspace{0.4cm}t\in[0,1]
\end{align*}
which shows that the sequence $\{x_{n}\}_{n=m}^{\infty}$ is uniformly bounded on $[0,1]$. Moreover, since $G(t,\tau)$ is uniformly continuous on $[0,1]\times[0,1]$, by Lebesgue dominated convergence theorem, the sequence $\{x_{n}\}_{n=m}^{\infty}$ equicontinuous on $[0,1]$. Thus by Arzela Ascoli Theorem the sequence $\{x_{n}\}_{n=m}^{\infty}$ is relatively compact and consequently there exist a subsequence $\{x_{n_{k}}\}_{k=1}^{\infty}$ converging uniformly to $x\in C[0,1]$. Moreover, in view of \eqref{xnt}, we have
\begin{align*}
x_{n_{k}}(t)=\int_{0}^{1}G(t,\tau)f\left(\tau,x_{n_{k}}(\tau)+\frac{1}{n_{k}}\right)d\tau
\end{align*}
as $k\rightarrow\infty$, we obtain
\begin{equation}\label{intsol}
x(t)=\int_{0}^{1}G(t,\tau)f(\tau,x(\tau))d\tau
\end{equation}
which in view of Lemma \ref{lemir}, leads to
\begin{align*}
{}^{C}D_{0^{+}}^{\,\mu}\,x(t)+f(t,x(t))=\omega\,x(t),\hspace{0.4cm}x'(0)=0,\hspace{0.4cm}x(1)=0.
\end{align*}
Also, ${}^{C}D_{0^{+}}^{\,\mu}\,x\in C(0,1)$. Further, from \eqref{intsol} in view of $(A2)$ and Lemma \ref{gbound}, we have
\begin{align*}
x(t)=\int_{0}^{1}G(t,\tau)f(\tau,x(\tau))d\tau\geq\int_{0}^{1}G(t,\tau)\,\gamma_{_{R}}d\tau=\frac{\gamma_{_{R}}\,\sigma_{\mu,\omega}(t)}{\omega\,E_{\mu,1}(\omega)}
\end{align*}
which shows that $x(t)>0$ for $t\in[0,1)$. Hence $x\in X$ is a positive solution of BVP \eqref{mp}.
\end{proof}

\begin{ex}

\begin{equation}\label{ebvp}\begin{split}
{}^{C}D_{{0}^{+}}^{\,^{1.9}}\,x(t)+\frac{\lambda}{\sqrt{\sigma_{_{1.9,\,2.0}}(t)\,\,\sigma_{_{1.9,\,2.0}}(1-t)}}\left(\frac{1}{\sqrt[5]{x(t)}}-x(t)+R\right)&=2\,x(t),\hspace{0.4cm}t\in(0,1)\\
x'(0)=0,\,\,x(1)&=0
\end{split}\end{equation}
where
\begin{align*}
0<\lambda<\min\left\{\frac{R^{\frac{6}{5}}}{13.3352\,(1+2R^{\frac{6}{5}})^{\frac{5}{4}}},3.59596\times R^{\frac{6}{5}}\right\}
\end{align*}
Here
\begin{align*}
f(t,x)=\frac{\lambda}{\sqrt{\sigma_{_{1.9,\,2.0}}(t)\,\,\sigma_{_{1.9,\,2.0}}(1-t)}}\left(\frac{1}{\sqrt[5]{x}}-x+R\right)
\end{align*}
Choose
\begin{align*}
q(t)=\frac{\lambda}{\sqrt{\sigma_{_{1.9,\,2.0}}(t)\,\,\sigma_{_{1.9,\,2.0}}(1-t)}},\hspace{0.4cm}u(x)=\frac{1}{\sqrt[5]{x}},\hspace{0.4cm}v(x)=x+R,\hspace{0.4cm}\gamma_{r}=1.94308\times\frac{\lambda}{\sqrt[5]{r}}
\end{align*}
Then,
\begin{align*}
\int_{0}^{1}q(t)dt=3.07853\times\lambda,\hspace{0.4cm}\int_{0}^{1}q(t)u(c\,\,\sigma_{_{1.9,\,2.0}}(t))dt=4.37043\times\frac{\lambda}{\sqrt[5]{c}}
\end{align*}
Moreover,
\begin{align*}
|f(t,x)|\leq q(t)(u(x)+v(x)),\text{ for }t\in(0,1),\hspace{0.4cm}x\in(0,\infty),
\end{align*}
\begin{align*}
f(t,x)\geq \gamma_{_{R}}\text{ for }t\in(0,1),\hspace{0.4cm}x\in(0,R].
\end{align*}
Further,
\begin{align*}
\frac{R}{E_{_{1.9,1.9}}(2)\,\chi_{_{R}}\left(1+\frac{v(R)}{u(R)}\right)}=\frac{R}{7.94329\times\sqrt[25]{R}\,\left(1+2R^{\frac{6}{5}}\right)\lambda^{\frac{4}{5}}}>1
\end{align*}
where
\begin{align*}
\chi_{r}=5.21001\times\lambda^{\frac{4}{5}}\times\sqrt[25]{r}
\end{align*}
Clearly, the assumptions $(A1)$ and $(A2)$ of Theorem \ref{mainth} are satisfied, therefore, the BVP \eqref{ebvp} has a positive solution $x\in X$.

\end{ex}


\begin{thebibliography}{99}

\bibitem{ars}R.P. Agarwal, D. O'Regan, S. Stanek, Positive solutions for Dirichlet problems of singular nonlinear fractional differential equations, J. Math. Anal. Appl. 371 (2010) 57-68.

\bibitem{bt}R.L. Bagley, P.J. Torvik, Fractional calculus in the transient analysis of viscoelastically damped structures, AIAA. J. 23 (1985) 918-925.

\bibitem{bbri}B. Bonilla, M. Rivero, L. Rodriguez-Germa, J.J. Trujillo, Fractional differential equations as alternative models to nonlinear differential equations, Applied Mathematical and Computation 187 (2007) 79-88.

\bibitem{borai}M.M. El-Borai, Semigroup and some nonlinear fractional differential equtions, Applied Mathematics and Computation 149 (2004) 823-831.

\bibitem{tsc}T.S. Chow, Fractional dynamics of interfaces between soft-nanoparticles and rough substrates, Phys. Lett. A. 342 (2005) 148-155.

\bibitem{dhelm} K. Diethelm, The Analysis of Fractional Differential Equations: An Application-Oriented Exposition Using Differential Operators of Caputo Type, Springer, 2010.

\bibitem{kilbas} A.A. Kilbas, H.M. Srivastava, J.J. Trujillo, Theory and Applications of Fractional Differential Equations, North-Holland Mathematics Studies, vol. 204, Elsevier Science B.V., Amsterdam, 2006.

\bibitem{kt} A.A. Kilbas, J.J. Trujillo, Differential equations of fractional order methods, results, problems, Appl. Anal., 78 (2001) 153-192.

\bibitem{lak4} V. Lakshmikantham, J. Vasundhara Devi, Theory of fractional differential equations in Banach spaces, European Journal of Pure and Applied Mathematics 1 (2008) 38-45.

\bibitem{luch} Y.F. Luchko, M. Rivero, J.J. Trujillo, M.P. Velasco, Fractional models, nonlocality and complex systems, Computers and Mathematics with Applications 59 (2010) 1048-1056.

\bibitem{ly} J.W. Lyons, J.T. Neugebauer, Positive solutions of a singular fractional boundary value problem with a fractional boundary condition, Opuscula Math. 37, no. 3 (2017), 421-434.

\bibitem{mmz} H. Maagli, N. Mhadhebi, N. Zeddini, Existence and estimates of positive solutions for some singular fractional boundary value problems, Abstr. Appl. Anal. (2014), Art. ID 120781.

\bibitem{miller} K.S. Miller, B. Ross, An Introduction to Fractional Calculus and Fractional Differential Equations, John Wiley and Sons, New York, 1993.

\bibitem{podlubny}I. Podlubny, Fractional Differential Equations. An Introduction to Fractional Derivatives, Fractional Differential Equations, to Methods of Their Solutions and Some of Their Applications, Mathematics in Science and Enginnering, vol. 198, Academic Press, San Diego, 1999.

\bibitem{maae}M.A.A. El-Sayeed, Fractional order diffusion wave equation, International Journal of Theoretical Physics 35 (1996) 311-322.

\bibitem{sz}A. Shi and S. Zhang, Upper and lower solutions method and a fractional differential equation boundary value problem, Electron. J. Qual. Theory Differ. Equ. 30 (2009) 1-13.

\bibitem{ss}S. Stanek, The existence of positive solutions of singular fractional boundary value problems, Comput. Math. Appl. 62 (2011) 1379-1388.

\bibitem{suzhang}X. Su, S. Zhang, Solutions to boundary value problems for nonlinear differential equations of fractional order. Elec. J. Diff. Equat. 26 (2009) 1-15

\bibitem{szhx}S. Sun, Y. Zhao, Z. Han, M. Xu, Uniqueness of positive solutions for boundary value problems of singular fractional differential equations, Inverse Probl. Sci. Eng. 20 (2012) 299-309.

\bibitem{xjy}X. Xu, D. Jiang, C. Yuan, Multiple positive solutions for the boundary value problem of a nonlinear fractional differential equation, Nonlinear Anal. 71 (2009) 4676-4688.

\bibitem{yjx}C. Yuan, D. Jiang, X. Xu, Singular positone and semipositone boundary value problems of nonlinear fractional differential equations, Math. Probl. Eng. (2009), Art. ID 535209.

\bibitem{zhang}S.-Q. Zhang, Positive solutions for boundary-valve problems of nonlinear fractional differential equations. Electron. J. Differential Equations 36 (2006) 1-12.

\bibitem{zmw}X. Zhang, C. Mao, Y. Wu, H. Su, Positive solutions of a singular nonlocal fractional order differential system via Schauder's fixed point theorem, Abstr. Appl. Anal. (2014), Art. ID 457965.

\bibitem{zshl}Y. Zhao, S. Sun, Z. Han, Q. Li, The existence of multiple positive solutions for boundary value problems of nonlinear fractional differential equations, Commun. Nonlinear Sci. Numer. Simul. 16 (2011) 2086-2097.

\bibitem{zshl1}Y. Zhao, S. Sun, Z. Han, Q. Li, Theory of fractional hybrid differential equations, Comput. Math. Appl. 62 (2011) 1312-1324.

\bibitem{zshl2} Y. Zhao, S. Sun, Z. Han, Q. Li, Positive solutions to boundary value problems of nonlinear fractional differential equations, Abstr. Appl. Anal. (2011) 1-16.

\bibitem{zshz} Y. Zhao, S. Sun, Z. Han, M. Zhang, Positive solutions for boundary value problems of nonlinear fractional differential equations, Appl. Math. Comput. 217 (2011) 6950-6958.

\end{thebibliography}
\end{document}